\documentclass[11pt]{article}
\usepackage{amsmath,amssymb,amsthm}
\usepackage[english]{babel}
\newcommand{\m}{\mathfrak m}

\newtheorem{thm}{Theorem}

\newtheorem{lemma}[thm]{Lemma}
\newtheorem{cor}[thm]{Corollary}
\newtheorem{prop}[thm]{Proposition}

\begin{document}

\title{{\bf Type sequences of one-dimensional
local analytically irreducible rings}}
\author{
{Valentina Barucci}\\
Dip.\! di Matematica,
Universit\`a di Roma ``La Sapienza'',\\
P.le Aldo Moro 2, Roma-00185, Italy\\
e-mail:\,{\tt barucci@mat.uniroma1.it}\\
{Ioana Cristina \c Serban}\\
e-mail:\,{\tt  ioanase@yahoo.com}}
\date{}
\maketitle

\begin{abstract}
We  extend the notion of type sequence to rings that are not 
necessarily residually rational. Using this invariant we characterize 
different types of rings as almost Gorenstein rings and rings of  maximal 
length.  
\end{abstract}

\section{Introduction}
\label{introduction}
Let $(R,\mathfrak m)$ be a one-dimensional local Cohen-Macaulay ring 
and let $\overline R$ be the integral closure of R in 
its field of quotients. If we assume that $R$ is 
analytically irreducible, i.e. that $\overline R$ is a DVR 
(with a valuation $v$) and a finitely generated $R$-module, 
then the values of the elements of $R$ form 
a numerical semigroup 
$v(R)= \{v(a)\ |\ a \in R, a \neq 0\}=\{s_0=0,s_1,\dots,
s_{r-1},s_r,\to\}$, where $s_0 < s_1< \dots <
s_r $ and any integer $x$, $x \geq s_r$ is in $v(R)$ 
and the conductor $C=(R: \overline R) $ is not zero.
 
If we further assume that $R$ is residually rational, 
i.e. that $k$, the residue field of $R$, is isomorphic to the 
residue field of $\overline R$, 
then $r=\ell_R(R/C)$ and a sequence of $r$ 
natural numbers $(t_1, \dots, t_r)$ 
is naturally associated to $R$, 
$t_i= \ell_R(\mathfrak a_i^{-1}/\mathfrak a_{i-1}^{-1})$, 
where $\mathfrak a_i=\{x\in R\arrowvert\;v(x)\geq s_i\}$. 
This sequence of natural numbers associated to the ring was 
for the first time considered by Matsuoka in \cite{mat}.  
As in \cite{bdofo} we call the sequence $(t_1, \dots, t_r)$ 
the type sequence of $R$, t.s.$(R)$ for short. 
In particular the length $t_1=\ell_R(\mathfrak m^{-1}/R)$ is 
the Cohen Macaulay type of $R$ and it turns out that 
$\sum_{i=1}^{r} t_i=\ell_R(\overline R/R)$. 
A typical example of an analytically irreducible and residually rational ring 
is the ring of an algebraic curve singlularity with one branch. 
 
It is well known that, for each one-dimensional local Cohen Macaulay ring 
with finite integral closure, the length $\ell_R(\overline R/R)$ 
is bounded under and above in the following way:
$$\ell_R(R/C)+t-1 \leq \ell_R(\overline R/R)\leq \ell_R(R/C)t$$
where again $C =(R: \overline R)$ and $t $ is the CM type of $R$. 

The first inequality depends on the existence of a canonical ideal 
(cf. \cite[Lemma 19 (e)]{bf}) and the second is proved in 
\cite[Theorem 1]{bh}. If the ring $R$ is Gorenstein, 
i.e. of CM type 1, then the two inequalities become equalities and 
$\ell_R(R/C)=\ell_R(\overline R/R)$. 
The rings which realize the minimal length for $\ell_R(\overline R/R)$, 
i.e. such that $\ell_R(R/C)+t-1 = \ell_R(\overline R/R)$ 
have been introduced in \cite{bf} with the name of almost Gorenstein 
and recently revealed an interest in a geometric context (cf.\cite{km}). 
On the other hand the rings which realize the maximal length for 
$\ell_R(\overline R/R)$ were characterized in \cite{bh} and also studied in \cite{de}.
 
In the analytically irreducible and residually rational case, 
there is a strict relation between the type sequence of $R$ and the length 
$\ell_R(\overline R/R)$. It is not surprising that the   almost Gorenstein rings 
are characterized by a type sequence of the form $(t,1,1, \dots,1)$ 
and those which realize the maximal length are characterized by a type 
sequence of the form $(t,t, \dots, t)$,  cf. \cite {bdofo} and \cite [Theorem 
1.7]{danna}. 
  
This paper deals with the analytically irreducible non-residually rational case. 
We have still a numerical semigroup $v(R)$ of values, 
but $k$, the residue field of $R$, is not isomorphic to $K$, 
the residue field of $\overline R$. 
The almost Gorenstein rings are characterized by a type sequence of 
the form $(t, n_1, \dots, n_{r+l})$ and the rings of maximal length   
by a  type sequence of the form $(t, tn_1, \dots, tn_{r+l})$, 
where $n_i$ are the dimensions 
of opportune $k$-vector subspaces of $K$. 

As usual, if $\mathfrak a$ and $\mathfrak b$ are fractional ideals of $R$, 
then $\mathfrak a :\mathfrak b:=\{ x \in Q(R)\ |\ x\mathfrak b \subseteq \mathfrak a\}$, 
where $Q(R)$ is the field of quotients of $R$, 
$\mathfrak a^{-1}=R:\mathfrak a$ and $\mathfrak a$ is divisorial if 
$R:(R:\mathfrak a))=\mathfrak a$.

\section{The result}
\label{an.ir.}

In all this paper  $R$ is a one-dimensional local analytically irreducible not necessarily residually rational ring. So the integral closure $\overline R$ is a DVR and $R$ has an associated semigroup of values:
 
\begin{equation}
v(R)=\{s_0=0,s_1,\dots,
s_{r-1},s_r=c,\to\},
\end{equation}
Denote by $X$ the generator of the maximal ideal of $\overline R$ and 
define the {\it conductor} of the ring as the
natural number $N$ such that $R:\overline R=X^N\overline R$. Note 
that $N\geq c$, thus we can set  $N=s_{r+l}=s_r+l=c+l$ for some $l\in
\mathbb N$. 
In the residually rational case 
we have  $N=c$.
Thus in order to extend  the definition of the type sequence to the non-residually rational case, some 
care is needed.
  As we shall see, in the general case the ``right''
definition will consist of a sequence of $r+l$ numbers. 

Let us see the details. Consider the ideals of $R$ defined as
\begin{equation}
\mathfrak a_i=\{x\in
R\arrowvert\;v(x)\geq s_i\},\;i\in \{0,\dots,r+l\}.
\end{equation}
It is evident that $\mathfrak a_0=R$, $\mathfrak a_1=\m$ and $\mathfrak
a_{r+l}=R:\overline R$. Moreover, we have the following chain of inclusions:
\begin{equation}
\mathfrak a_{r+l}\subset \dots \subset \mathfrak a_0=\mathfrak 
a_0^{-1}=R\subseteq \mathfrak
a_1^{-1} \dots \subseteq \mathfrak a_{r+l}^{-1},
\end{equation}
 Note that whereas on the left side we have strict inclusions, on the right 
side, {\it a priori}, some of the inclusions could be equalities.

  The following facts about the ideals ${\mathfrak a}_i$ are well known, but 
we recall them for the convenience of the reader: 
\begin{prop}
For every $i\in \{0,\dots,r+l\}$, the ideals defined above
have the following properties.
\begin{itemize}
\item[1.] $\mathfrak a_{r+l}^{-1}=\overline R$;
\item[2.] $\mathfrak a_i$ is divisorial.
\item[3.] If $i>0$ then $\mathfrak a^{-1}_{i}\neq \mathfrak a^{-1}_{i-1}$ 
and hence $\ell_R(\mathfrak a_i^{-1}/\mathfrak a_{i-1}^{-1})\geq 1$.
\end{itemize}
\end{prop}
\begin{proof} 
$1.$ As $\mathfrak a_{r+l}=X^N\overline R$, we have
\begin{eqnarray}\nonumber
\mathfrak a_{r+l}^{-1} &=& R:X^N\overline R=X^{-N}(R:\overline R)\\
&=& X^{-N}X^N\overline R=\overline R.
\end{eqnarray} 

$2.$ As $\overline R=R:\mathfrak a_{r+l}^{-1}$, we have that 
$\overline R$ is divisorial as a fractional ideal of $R$. It follows 
that $X^h\overline R$ is divisorial for every $h\in \mathbb N$.
This shows that $\mathfrak a_i$ is divisorial, as
\begin{equation}
\mathfrak a_i=R\cap X^{s_i}\overline R.
\end{equation}

$3.$ If $i>0$, then both $\mathfrak a_i$ and $\mathfrak a_{i-1}$ are divisorial.
Thus, if   $R:\mathfrak a_{i-1}= R:\mathfrak a_i$,  then $\mathfrak a_{i-1}=
R:(R:\mathfrak a_{i-1})=R:(R:\mathfrak a_i)=\mathfrak a_i$, which is
contradiciton.
\end{proof}

Now we are ready to give our definition of type sequence of the ring 
$R$. For every $i\in \{1,2,\dots,r+l\}$ let 
\begin{equation}
t_i(R):=\ell_R(\mathfrak a_i^{-1}/\mathfrak a_{i-1}^{-1}).
\end{equation}
We call the sequence of numbers $(t_1(R),t_2(R),\dots,t_{r+l}(R))$ the 
{\it type sequence} of $R$, and we denote it by ${\rm t.s.}(R)$.

As in the residually rational case we have that
\begin{equation}
t_1(R)=\ell_R(\m^{-1}/R)=: t(R)
\end{equation}
which is the Cohen-Macaulay type of $R$.

Note that for every $1\leq i \leq r+l$, $\mathfrak m 
{\mathfrak a}_{i-1}\subseteq {\mathfrak a}_i$ and so $\mathfrak 
a_{i-1}/\mathfrak a_i$, is a $k$-vector space; let us  
denote it by $V_R(s_{i-1})$). Since the inclusion ${\mathfrak a}_{i } \subset  {\mathfrak a}_{i-1}$ is strict, 
$V_R(s_{i-1})\neq {0}$ and hence the number  $n_{i-1}:={\rm
dim}_kV_R(s_{i-1})$ is a positive integer. These vector spaces were 
considered also in \cite {cdk} and can be defined not only for the ring $R$ 
but also for 
any fractional ideal of  $R$. Let $F$ be such an 
ideal and $i\in \mathbb N$. Then
\begin{equation}
F(i):=\{x\in F|v(x)\geq i\}
\end{equation}
is a fractional ideal of $R$  and we have 
$F(i)\subseteq F(j)$ for every $i\geq j$. The $R-$modules $F(i)/F(i+1)$ 
are also vector spaces over $k$, and we denote them by  $V_F(i)$. 

As we have outlined in the introduction,   these vector spaces 
are very important for studying lengths for the analytically ireducible rings 
which are not residually 
rational. If $E \subseteq F$ are fractional ideals of $R$, then, in the residually rational case, $\ell_R(F/E)=\#\{ v(F) \setminus v(E)\} $, cf. \cite[Proposition 1] {mat}. In the non-residually 
rational case we use the dimensions of the previous defined  vector spaces 
as it was proved in 
\cite{s}. 
 
\begin{prop}{\rm \cite[Proposition 11]{s}} 
\label{prop:lungnotres}
Let $E$ and $F$ be two fractional ideals of $R$ such that  
$E\subseteq F \subseteq \overline R$. Then there exists an $s \in \mathbb N$
such that
\begin{equation}
\ell_R(F/E)=\sum_{r=0}^{s-1}[{\rm dim}_k(V_F(r))-{\rm dim}_k(V_E(r))].
\end{equation} 
\end{prop}
We recall also another  result which in this form appears in \cite{s} and 
in fact it is an adapted version  of \cite[Proposition 3.5] {cdk}. Observe that if $V$ and $W$ are two $k$-vector subspaces of $K$, where $k \subseteq K$ is a field extension, then $(V:W):=\{x \in K\ |\ xW\subseteq V\}$ is also a $k$-vector subspace of $K$.  
\begin{lemma}{\rm \cite[Lemma 3]{s}} 
\label{lem:dim}
Let $k\subseteq K$ be an extension of fields with 
$n={\rm dim}_{k}K<\infty$ 
and let $V\subset K$ be  an $n-1$-dimensional $k$-vector subspace of $K$. Then 
for every $k$-vector subspace $W\subseteq K$ we have 
\begin{equation}
{\rm dim}_k(V:W)+{\rm dim}_k(W)=n.
\end{equation}
\end{lemma}
In order to prove our main theorem, we need the next result on the 
dimensions of 
previous defined vector spaces related to the fractional 
ideals $\mathfrak a_i^{-1}$, $i\in \{1, \dots r+l\}$.
\begin{lemma}
\label{lem:aidim}
Let $n={\rm dim}_kK$, where $k$ is the residue field of $R$ and $K$ is the residue field of $\overline R$. Then:
\begin{itemize}
\item[1.]
${\rm {dim}}_k(V_{\mathfrak a_{i}^{-1}}(N-1-s_{i-1}))=n$;
\item[2.]
${\rm {dim}}_k(V_{\mathfrak a_{i-1}^{-1}}(N-1-s_{i-1}))\leq n-n_{i-1}$;
\end{itemize}
\end{lemma}
\begin{proof}
Let us prove the first assertion.
Fix an $i\in \{1,\dots,r+l\}$. Then for any $\gamma \in K$ we have that 
\begin{equation}  
\label{eqn:tslow2}
\gamma X^{N-1-s_{i-1}}\mathfrak {a_{i}}\subseteq
X^{N-1+(s_i-s_{i-1})}\overline{R} \subseteq X^N\overline{R}\subseteq R.
\end{equation}
Therefore we have 
\begin{equation}  
\label{eqn:tslow3}
\gamma X^{N-1-s_{i-1}}\in R :\mathfrak {a_{i}}=\mathfrak {a_i}^{-1},
\end{equation}
for every $\gamma \in K$. Thus $KX^{N-1-s_{i-1}}\subseteq  R:
\mathfrak {a_{i}}$. It  follows that $V_{\mathfrak 
{a_i}^{-1}}(N-1-s_{i-1})\simeq K$ and this is of dimension $n$ over $k$. 

Now we can prove the second assertion.
It is easy to see that 
\begin{equation}
\label{eqn:tslow1}
\gamma X^{N-1-s_{i-1}} {\mathfrak {a_{i-1}}}\subseteq R
\Longleftrightarrow
\gamma V_R(s_{i-1}) \subseteq V_R(N-1)
\end{equation}
which is of course further equivalent to
$\gamma\in {V_R(N-1):V_R(s_{i-1})}$.

Thus we have that:
\begin{equation}
\label{eqn:tslow2}
\gamma \in V_{\mathfrak a_{i-1}^{-1}}(N-1-s_{i-1}) 
\Longleftrightarrow
\gamma \in ( V_R(N-1):V_R(s_{i-1}))
\end{equation}

Then we can conclude that:
\begin{equation}
{\rm dim}_k(V_{\mathfrak {a_{i-1}^{-1}}}(N-1-s_{i-1}))={\rm
dim}_k({V_R(N-1)}:{V_R(s_{i-1})}).
\end{equation}

As ${V_R(N-1)}$ is a proper subspace of $K$, we can find a $k$-vector subspace $U\subset
K$ of codimension $1$
such that ${V_R(N-1)} \subseteq U$ and so
\begin{eqnarray}\nonumber
{\rm dim}_k({V_R(N-1)}:{V_R(s_{i-1})}) &\leq& {\rm
dim}_k(U:{V_R(s_{i-1})})\\
\nonumber &=&  n-{\rm dim}_k{(V_R(s_{i-1})})=\\
&& =n-n_{i-1},
\end{eqnarray} 
where for the first equality 
we have used  Lemma \ref{lem:dim}.

\end{proof}

We shall see now  certain upper  and lower bounds on $t_i(R)$, which generalize \cite[Proposition 3]{mat}. 
\begin{prop}
\label{thm:tsnonresid}
For every 
$i \in
\{1,\dots,r+l\}$:
\begin{equation}
n_{i-1}\leq t_i(R)\leq t(R)\;n_{i-1}.
\end{equation}
\end{prop}
\begin{proof}

For showing the upper bound we shall use \cite[Lemma 1]{bh}. This affirms 
that, if we have two ideals of the ring $R$, $I_1$ and $I_2$ such that 
$I_1\subseteq I_2$ and $I_2/I_1$ is a simple $R-$module, then 
\begin{equation}
\ell_R(R:I_1/R:I_2)\leq t(R).
\end{equation}
We are using this result for our ideals $\mathfrak a_{i}\subseteq 
\mathfrak a_{i-1}$. The $R-$module $\mathfrak a_{i-1}/\mathfrak a_{i}$ is 
not 
simple, but it is in fact a $k-$vector space of finite dimension equal to  
$n_{i-1}$. Then we can apply \cite[Lemma 1]{bh}  $n_{i-1}$ times and we 
conclude the proof for the  upper bound.

 Now we want to show the lower bound. Using 
Proposition \ref{prop:lungnotres} we have that
\begin{equation}
\label{eqn:typeslow}
\ell_R(\mathfrak {a_i^{-1}}/\mathfrak {a_{i-1}^{-1}})
\geq {\rm dim}_k(V_{\mathfrak
{a_i^{-1}}}(N-1-s_{i-1}))-{\rm dim}_k(V_{\mathfrak {a_{i-1}^{-1}}}(N-1-s_{i-1})).
\end{equation}
By Lemma \ref{lem:aidim}, we get:
\begin{equation}
t_i(R):=\ell_R(\mathfrak {a_i^{-1}}/\mathfrak {a_{i-1}^{-1}})
\geq n-(n-n_{i-1})=n_{i-1}.
\end{equation}
\end{proof}

Similarly to the residually rational case,  
we can characterize  rings of minimal and maximal length by their type sequences.
\begin{thm}
\label{thm:ringschar}
Let $n_i={\rm dim}_k(\mathfrak a_i /\mathfrak a_{i-1})$. Then

\begin{itemize}
\item[1.]
$R$ is almost Gorenstein if and only if
\begin{equation}
{\rm t.s.}(R)=(t(R),n_1,n_2,\dots,n_{r+l-1}).
\end{equation}
\item[2.]
$R$ is of maximal length if and only if
\begin{equation}
{\rm t.s.}(R)=(t(R),t(R)n_1,t(R)n_2,\dots,t(R)n_{r+l-1}).
\end{equation}
\end{itemize}
\end{thm}
\begin{proof}
Recall that the almost Gorenstein property means (cf. 
\cite[Definition-Proposition 20]{bf})
that
\begin{equation}
\label{eqn:almGor}
\ell_R(\overline R/R)=\ell_R(R/\mathfrak
a_{r+l})+t(R)-1.
\end{equation}
Equation (\ref{eqn:almGor}) is equivalent to:
\begin{equation}
\sum_{i=1}^{r+l}t_i(R)=\sum_{i=0}^{r+l-1}n_i+t(R)-1
\end{equation}
As $n_0={\rm dim}_k(V_R(s_0))={\rm dim}_k(R/\m)=1$ and $t_1(R)=t(R)$,
the previous
equation is equivalent to:
\begin{equation}  
\sum_{i=2}^{r+l}t_i(R)=\sum_{i=1}^{r+l-1}n_i.
\end{equation}
Thus, the conclusion follows using Proposition \ref {thm:tsnonresid} which 
claims that $n_{i-1}\leq t_i$, so the previous equality holds if and only 
if 
$t_i=n_{i-1}$ for every $i$, $2\leq i\leq r+l$.

The ring $R$ is of maximal length if and only if
\begin{equation}
\label{eqn:maxleng}
\ell_R(\overline R/R)=t(R)\ell_R(R/\mathfrak
a_{r+l}).
\end{equation}
And this is further equivalent to
\begin{equation}
\sum_{i=1}^{r+l}t_i(R)=t(R)\sum_{i=0}^{r+l-1}n_i=\sum_{i=0}^{r+l-1}t(R)n_i.
\end{equation}
By Proposition \ref {thm:tsnonresid}, $t_i\leq t(R)n_{i-1}$ for every 
$i\in\{1,\dots, r+l-1\}$, and we know also that $n_0=1$, so the previous 
equality 
holds if and only if $t_i=t(R)n_{i-1}$ for every $i\in \{1,\dots, r+l\}$.
\end{proof}
As a consequence of   Theorem \ref{thm:ringschar},1, and of the 
fact 
that
a Gorenstein (Kunz) ring is an almost Gorenstein ring of type $1$
($2$, respectively), we can draw the following conclusion.
\begin{cor}
 
\begin{itemize}
\item[1.]
$R$ is Gorenstein  if and only if
${\rm t.s.}(R)=(1,n_1,\dots,n_{r+l-1})$,
\item[2.]
$R$ is Kunz  if and only if
${\rm t.s.}(R)=(2,n_1,\dots,n_{r+l-1})$.
\end{itemize}
\end{cor}
Note that the proof for Gorensteiness could have been also obtained in a direct
manner. Indeed, if $R$ is Gorenstein, then the ring itself is a canonical ideal,
and so for every $i\in \{1,\dots,r+l\}$, 
\begin{eqnarray}\nonumber t_i(R) &=&
\ell_R((R:\mathfrak {a_i})/R:\mathfrak {a_{i-1}})\\ \nonumber &=&
\ell_R((R:(R:\mathfrak {a_{i-1}}))/(R:(R:\mathfrak {a_{i}})))=\ell_R(\mathfrak
{a_{i-1}}/\mathfrak {a_i})\\ &=& {\rm dim}_kV_R(s_{i-1})=n_{i-1}. 
\end{eqnarray}

\section{Examples}

{\bf Example 1}. Consider the following subring of the ring of power series $\mathbb Q(\sqrt2,\sqrt3)[[X]]$. 

$$R= \mathbb Q+X^3\mathbb Q(\sqrt2,\sqrt3)+X^4\mathbb Q(\sqrt2,\sqrt3)+X^5\mathbb Q(\sqrt2 )+X^6\mathbb Q(\sqrt2,\sqrt3)[[X]]$$

In this example, $k=\mathbb Q$ and $K= \mathbb Q(\sqrt2,\sqrt3)$. According to the notation of previous section, we have $n_0=1, n_1=4, n_2=4, n_3=2$. Moreover, since

 $\mathfrak a_1^{-1}= \mathbb Q(\sqrt2 )+X^3\mathbb Q(\sqrt2,\sqrt3)[[X]]$
 
 $\mathfrak a_2^{-1}= \mathbb Q(\sqrt2 )+X^2\mathbb Q(\sqrt2,\sqrt3)[[X]]$,
 
 $\mathfrak a_3^{-1}= \mathbb Q(\sqrt2 )+X \mathbb Q(\sqrt2,\sqrt3)[[X]]$,
 
 $\mathfrak a_4^{-1}=  \mathbb Q(\sqrt2,\sqrt3)[[X]]$,
 
 we get $t=t_1=3, t_2=4, t_3=4, t_4=2$, so that 
 the  type sequence is $(t,n_1,n_2,n_3)$ and the ring is almost Gorenstein.

\bigskip

{\bf Example 2} A ring of maximal length.

$$R=\mathbb R+X^3i\mathbb R+X^6 \mathbb R+X^9\mathbb C[[X]]$$

Here $k= \mathbb R$, $K= \mathbb C$, $n_0=n_1=n_2=1$ and since

 $\mathfrak a_1^{-1}=\mathbb R+X^3i\mathbb R+X^6 \mathbb C[[X]]$
 
  $\mathfrak a_2^{-1}=\mathbb R+X^3  \mathbb C[[X]]$

 $\mathfrak a_3^{-1}=\mathbb C[[X]]$

we get $t=t_1=5, t_2=5, t_3=5$, so that the type sequence is $(t, tn_1,tn_2)$ and the ring is of maximal length.

\bigskip

The examples above are generalized semigroup rings, GSR for short,  
i.e. rings of the form
$$k+XV_1+ \dots+X^{N-1}V_{N-1}+X^NK[[X]]$$
where $V_i$ are $k$-vector subspaces of $K$. 
To every one-dimensional analytically irreducible ring $R$ can be associated a GSR $\widetilde R$,
as in \cite{s}. More precisely  $\widetilde R$ is the subring of $K[[X]]$ defined, with the notation of the previous section,  as
$$ \widetilde R:= \sum_{i \geq0} V_R(i)X^i$$
That is in fact the generalization of the way of associating 
to $R$, in the residually rational case, 
the semigroup ring $k[[S]]$, where $S=v(R)$. 
Observe however that the type sequence of $R$ and its associated GSR 
is not always the same. For example, if $R=k[[X^4,X^6+X^7,X^{10}]]$,
with characteristic of $k$ unequal to $2$, then the associated GSR is 
$k[[X^4,X^6,X^{11},X^{13}]]$, which has type sequence $(3,1,1,1)$. 
On the other hand the type sequence of $R$ is 
$(2,2,1,1)$, in fact $\mathfrak a_1^{-1}$ contains no 
element with value 2, 
but $\mathfrak a_2^{-1}$ contains $X^2-X^3$.     
   However we can prove that: 
 
\begin{prop}
The ring $R$ is   almost Gorenstein if and only if 
the associated GSR $\widetilde R$ is almost Gorenstein and ${\rm 
{type}}(R)={\rm {type}}(\widetilde R)$.
\end{prop}
\begin{proof} Let $\omega$ be a canonical ideal of $R$, $R \subseteq \omega \subseteq \overline R$ (cf. \cite{bf} for the definition and the existence). Then, by \cite[Theorem 17]{s},     $\widetilde \omega= \sum_{i \geq 0} V_{\omega} (i)X^i$       is a canonical ideal of $\widetilde R$. Moreover, by \cite [equation (24) and Lemma 16]{s}, we have:
\begin{equation} \ell_R(\omega/R)= \ell_{\widetilde R}(\widetilde \omega/\widetilde R) \geq {\rm type}(\widetilde R)-1 \geq {\rm type}( R)-1\end{equation}
By \cite[Definition-Proposition 20]{bf} we get that $R$ is almost Gorenstein if and only if $\ell_R(\omega/R)  \geq {\rm type}(  R)-1$. Thus  $R$ is almost Gorenstein if and only if both inequalities of (29) are equalities, that is if and only if $\widetilde R$ is almost Gorenstein  and  ${\rm 
{type}}(R)={\rm {type}}(\widetilde R)$.
 \end{proof}

We shall give now an example of computing the type sequence of a ring which 
is not a GSR and the length $\ell_R( \overline R/R)$ is not minimal nor maximal.

\bigskip

{\bf Example 3} Let $R=\mathbb R[[iX^3+X^4,X^5,iX^{10}+X^{11},X^{16}]]$. 
For this ring $k=\mathbb R$ and $K=\mathbb C.$ 
Observe that $R$ is not a GSR, but we can compute its associated GSR.
First let us try to compute the type sequence of $R$.
After some computations we can write $R$ as:\\
$R=\mathbb R+(iX^3+X^4)\mathbb R+X^5\mathbb R+(-X^6+2iX^7+X^8)\mathbb 
R+(iX^8+X^9)\mathbb R
+(-iX^9+3iX^{11}+X^{12})\mathbb R+X^{10}\mathbb 
R$+
$(iX^{10}+X^{11})\mathbb 
R+(-X^{11}+2iX^{12})\mathbb R
+(X^{12}-2X^{14})\mathbb R+X^{13}\mathbb R+(iX^{13}+X^{14})\mathbb 
R$+
$iX^{14}\mathbb R
+X^{15}\mathbb C[[X]]$.\\
Thus for the ring $R$ we have: the conductor of the ring $N=15$ and 
$n_0=1$, $n_1=1$, $n_2=1$, $n_3=1$, 
$n_4=1$, $n_5=1$, $n_6=2$, $n_7=1$, $n_8=1$, $n_9=2$, $n_{10}=1$.  
We compute now the inverses of the ideals which appear in the definition of the 
type sequence.\\
$\mathfrak m^{-1}=\mathbb R+(iX^3+X^4)\mathbb R+X^5\mathbb R+
(-X^6+2iX^7+X^8)\mathbb R+(X^7-2X^9)\mathbb R+(iX^8+X^9)\mathbb R
+(-iX^9+3iX^{11}+X^{12})\mathbb R+X^{10}\mathbb R+(iX^{10}+X^{11}\mathbb R
+X^{11}\mathbb R+X^{12}\mathbb C[[X]]$. \\
$\mathfrak a_2^{-1}=\mathbb R+(iX^3+X^4)\mathbb R+X^5\mathbb R+
(-X^6+2iX^7+X^8)\mathbb R+(X^7-2X^9)\mathbb R
+(iX^8+X^9)\mathbb R+X^9i\mathbb R+X^{10}\mathbb C[[X]]$.\\ 
$\mathfrak a_3^{-1}=\mathbb R+(iX^3+X^4)\mathbb R+X^5\mathbb R+
(-X^6+2iX^7+X^8)\mathbb R+X^7\mathbb R+(iX^7+X^8)\mathbb R
+X^8i\mathbb R+X^9\mathbb C[[X]]$.\\
$\mathfrak a_4^{-1}=\mathbb R+(iX^3+X^4)\mathbb R+X^5\mathbb R
+X^6\mathbb R+X^7\mathbb C[[X]]$.\\
$\mathfrak a_5^{-1}=\mathbb R+(iX^3+X^4)\mathbb R+X^5\mathbb R
+X^6\mathbb C[[X]]$.\\
$\mathfrak a_6^{-1}=\mathbb R+(iX^3+X^4)\mathbb R+X^5\mathbb C[[X]]$.\\
$\mathfrak a_7^{-1}=\mathbb R+(iX^2-X^3)\mathbb R
+X^3i\mathbb R+X^4\mathbb C[[X]]$\\
$\mathfrak a_8^{-1}=\mathbb R+X^2i\mathbb R
+X^3\mathbb  C[[X]]$.\\
$\mathfrak a_9^{-1}=\mathbb R+X^2\mathbb C[[X]]$.\\
$\mathfrak a_{10}^{-1}=\mathbb R+X\mathbb C[[X]]$.\\
$\mathfrak a_{11}^{-1}=\mathbb C[[X]]=\overline R$.\\
Then: $t_1=\ell_R(\mathfrak m^{-1}/R)=3$,
$t_2=1$, $t_3=2$, $t_4=1$,  $t_5=1$, $t_6=1$, $t_7=3$,
$t_8=1$, $t_9=1$, $t_{10}=2$, $t_{11}=1$.\\
The associated GSR of $R$ is $\widetilde R=\mathbb 
R[[iX^3,X^5,iX^{10},iX^{17}]]$.
After computations we have that the type sequence of $\widetilde R$ is 
$(t_1=3,t_2=1,t_3=2,t_4=1,t_5=1,t_6=1,t_7=3,t_8=1,t_9=1,t_{10}=2,t_{11}=1)$.
We observe that in this case the type sequence of $R$ is equal to the type 
sequence of its 
associated GSR.

\end{document}